\RequirePackage{fix-cm}
\documentclass[smallextended]{svjour3}       
\smartqed  
\usepackage{cite}
\usepackage{algpseudocode}
\usepackage[ruled]{algorithm}
 \usepackage{url}
\usepackage{booktabs}
\usepackage{graphicx}
\usepackage{tabularx}
\usepackage{siunitx}
\usepackage{color}
\usepackage[dvipsnames]{xcolor}
\usepackage{latexsym}
\usepackage{amsfonts}
\usepackage{amsmath}
\usepackage{float}
\usepackage{subfig}
\usepackage{multirow}
\usepackage{adjustbox}
\usepackage{mathtools}
\usepackage{enumerate}
\usepackage{xcolor}

\newcommand{\R}{\rm{I\!R}}

\def\square{{\setbox0=\hbox{X}\hbox to \ht0{\vrule\hss\vbox to \ht0{
  \hrule width \ht0\vfil\hrule width \ht0}\vrule}}}

\makeatletter

\DeclareMathOperator{\argmin}{argmin}

\title{Zero Order Stochastic Weakly Convex Composite Optimization}

\begin{document}

\title{Zero Order Stochastic Weakly Convex Composite Optimization}

\titlerunning{Zero Order Stochastic Weakly Convex Composite Optimization}        

\author{V. Kungurtsev\and
	F. Rinaldi 
}


\institute{
V. Kungurtsev \at
Department of Computer Science, Faculty of Electrical Engineering,\\
Czech Technical University in Prague
Prague, Czech Republic\\
 \email{vyacheslav.kungurtsev@fel.cvut.cz}\\
Research supported by
  the OP VVV project
CZ.02.1.01/0.0/0.0/16 019/0000765 ``Research Center for Informatics''      
 \and
F. Rinaldi \at
              Dipartimento di Matematica ``Tullio Levi-Civita'', Universit\`a di Padova \\
              Via Trieste, 63, 35121 Padua, Italy \\
              Tel.: +39-049-8271424\\
              \email{rinaldi@math.unipd.it}           
}
\date{Received: date / Accepted: date}

\maketitle

\begin{abstract}
In this paper we consider stochastic weakly convex composite problems, however
without the existence of a stochastic subgradient oracle. 
We present a derivative free algorithm that uses a two point
approximation for computing a gradient estimate of the smoothed 
function. We prove convergence at a similar rate
as state of the art methods, however with a larger constant, and report some numerical
results showing the effectiveness of the approach. 
\end{abstract}



\section{Introduction}
In this paper, we study the following class of problems:
\begin{equation}\label{eq:prob}
\min_{x\in\mathbb{R}^n} \, \phi(x) := f(x)+r(x),
\end{equation}
where it is assumed that,
\begin{enumerate}
\item $f(\cdot)$ is $\rho$-weakly convex, i.e., $f(x)+\rho\|x\|^2$ is convex for some $\rho>0$, and locally Lipschitz with constant $L_0$.
\item $f(\cdot)$ is nonsmooth, as it is not necessarily continuously differentiable.
\item The subgradients $\partial f$ are not available. Furthermore, function evaluations $f(x)$ are not available,
but rather noisy approximations thereof. Thus we are in the noisy or stochastic zero order/derivative free optimization setting.
We thus write 
$$f(x)=\mathbb{E}_\xi[F(x;\xi)]=\int_{\Xi} F(x,\xi)dP(\xi),$$ 
with $\{F(\cdot, \xi), \ \xi\in \Xi\}$ a collection of real valued functions and $P$ a probability distribution over the set $\Xi$ to be precise.
\item $r(\cdot)$ is convex (but not necessarily continuously differentiable) and simple.
\end{enumerate}
One standard subset of composite functions is given by $f(x)=h(c(x))$ where $h$ is nonsmooth and convex
and $c(x)$ is continuously differentiable but non-convex (see, e.g., \cite{davis2019stochastic} and references therein).

We further note that the $\rho$-weak convexity property for a given function $f$ is equivalent to hypomononicity of its subdifferential map, that is
\begin{equation}\label{eq:hypocoer}
\langle v-w,x-y\rangle\ge-\rho\|x-y\|^2 
\end{equation}
for $v\in\partial f(x)$ and $w\in\partial f(y)$.


A canonical exact method for solving a weakly convex stochastic optimization problems is given as repeated
iterations of,
\begin{equation}\label{eq:proxmethod}
x_{k+1} := \text{argmin}_y \left\{ f_{x_k}(y;S_k)+r(y)+\frac{1}{2\alpha_k}\|y-x_k\|^2\right\}
\end{equation}
where $\alpha_k>0$ is a stepsize sequence, typically taken to satisfy $\alpha_k\to 0$, and
$f_{x_k}(y;S_k)$ is approximating $f$ at $x_k$ using a noisy estimate $S_k$ of the data. 
A basic stochastic subgradient method will use the linear model $$f_{x_k}(y;S_k) =f(x_k) +\xi^T (y-x_k)$$ where 
$\xi\approx \bar \xi \in \partial f(x_k)$. When using this approach, it is common to consider
the existence of  some oracle of an unbiased estimate of an element of the subgradient that enables one to build up the approximation $f_{x_k}$ with favorable properties (see,e.g., \cite{davis2019stochastic}
or~\cite{duchi2018stochastic}).
In our case we assume such an oracle is not available, and we only get access, at a point $x$, to a noisy function value observation $F(x,\xi)$.
Stochastic problems with only functional information available often arise in optimization, machine learning and statistics.
A classic example is simulation based optimization (see,e.g., \cite{amaran2016simulation,larson2019derivative} and references therein), where function evaluations usually represent the experimentally obtained 
behavior of a system and in practice are given by means of specific simulation tools, 
hence no internal or analytical knowledge for the functions is provided. Furthermore, evaluating the function at a given point is in many cases a
computationally expensive task, and only a limited budget of evaluations is available in the end. 
In machine learning and statistics, a widely studied problem is bandit optimization, where a player and an adversary compete, with the player strategy expressed using the variable vector $x$ 
and the adversary strategy given by $\xi$ (the player has no control over this set of variables), and player's loss function is $F(x,\xi)$. 
The goal is choosing the optimal strategy based only on observations of the function values (see, e.g., Section 4.2 in \cite{larson2019derivative} for further details).
Recently, suitable derivative free/zero order optimization methods have been proposed for handling stochastic functions (see,e.g., \cite{blanchet2019convergence, chen2018stochastic, duchi2015optimal, larson2016stochastic}).
For a complete overview of stochastic derivative free/zero order  methods, we refer the interest reader to the recent review \cite{larson2019derivative}.
The reason why we focus on weakly convex functions is that such a class of function shows up in the modeling of many different 
real world applications like, e.g., (robust) phase retrieval, sparse dictionary learning, conditional value at risk (see \cite{davis2019stochastic} for a complete description of those problems) and there might be cases
where we only get access, at a point $x$, to an unbiased estimate of the loss function $F(x,\xi)$. We thus need to resort to a stochastic derivative free/zero order approach in order to handle our problem.

At the time of writing, zero order, or derivative free optimization for weakly convex problems has not been
investigated. There are a number of works for stochastic nonconvex zero order optimization 
(e.g.,~\cite{balasubramanian2019zeroth}) and
nonsmooth convex derivative free optimization (e.g.,~\cite{duchi2015optimal}).

In the case of stochastic weakly convex optimization but with access to a noisy element of the
subgradient, there are a few works that have appeared fairly recently. 
asymptotic convergence was shown in~\cite{duchi2018stochastic}, which proves convergence with probability one
for the method given in~\eqref{eq:proxmethod}. Non-asymptotic convergence, as in convergence rates in expectation, 
is given in the two 
papers~\cite{davis2019stochastic} and~\cite{li2019incremental}. 

In this paper, we follow the approach proposed in \cite{duchi2015optimal} to handle nonsmoothness in our problem. We consider a smoothed 
version of the objective function, and we then apply a two point strategy to estimate its gradient. This tool is thus embedded in a proximal 
algorithm similar to the one described in \cite{davis2019stochastic} and enables us to get convergence at a similar rate as the original method
(although with larger constants).

The rest of the paper is organized as follows. In Section~\ref{s:alg} we describe the algorithm and provide some preliminary lemmas needed
for the subsequent analysis. Section~\ref{s:convergence} contains the convergence proof. In Section~\ref{s:numerics} we show some numerical
results on two standard test cases. Finally we conclude in Section~\ref{s:conclusion}.

\section{Two Point Estimate and Algorithmic Scheme}\label{s:alg}

We use the two point estimate presented in~\cite{duchi2015optimal} to generate an approximation to an element of the subdifferential.
In particular, consider the smoothing of the function $f$,
\[
f_{u_{1,t}}(x) =\mathbb{E}[f(x+zu_{1,t})] = \int_{z} f(x+zu_{1,t}) dz
\]
where $z$ is a standard normal variable, i.e., $z\sim \mathcal{N}(0,I_n)$. The two point estimate we use is given by,
\begin{equation}\label{eq:g}
\begin{split}
g_t &=G(x_t,u_{1,t},u_{2,t},Z_{1,t},Z_{2,t},\xi_t) =\\
& = \frac{F(x_t+u_{1,t} Z_{1,t}+u_{2,t} Z_{t,2};\xi_t)-F(x_t+u_{1,t} Z_{t,1};\xi_t)}{u_{2,t}} Z_{t,2},
\end{split}
\end{equation}
where $\{u_{1,t}\}_{t=1}^\infty$, $\{u_{2,t}\}_{t=1}^\infty$ are two nonincreasing sequences of positive parameters such that $u_{2,t}\leq u_{1,t}/2$,
$x_t$ is the given point, $\xi$ is the sample of the stochastic oracle, $Z_1\sim \mu_1$ and $Z_2\sim \mu_2$ are two vectors independently sampled from  
distributions $\mu_1\sim \mathcal{N}(0,I_n)$ and $\mu_2\sim \mathcal{N}(0,I_n)$. 

We now report a result that provides theoretical guarantees on the error in the estimate.
\begin{lemma}\cite[Lemma 2]{duchi2015optimal} The gradient estimator \eqref{eq:g} has expectation 
\begin{equation}\label{eq:ex1}
\mathbb{E}[g_t] = \nabla f_{u_{1,t}}(x)+\frac{u_{2,t}}{u_{1,t}} e(x,u_{1,t},u_{2,t}), 
\end{equation}
where $\|e(x,u_{1,t},u_{2,t})\|\le E$ for all $t$, with $E>0$ finite and real valued.

Furthermore, there exists a $G>0$ such that,
\begin{equation}\label{eq:ex2}
\mathbb{E}[\|g_t\|^2] \le G. 
\end{equation}
\end{lemma}

We recall some other useful facts. From~\cite{duchi2012randomized} it holds that,
\begin{equation}\label{eq:smoothacc}
f(x)\le f_{u_{1,t}}(x)\le f(x)+u_{1,t} \bar B \sqrt{n+2}
\end{equation}
with $\bar B>0$ a finite real valued parameter,  and that $f_{u_{1,t}}(x)$ is Lipschitz continuously differentiable with constant $\frac{L_0\sqrt{n}}{u_{1,t}}$, 
and
\begin{equation}\label{eq:boundnabfu}
\|\nabla f_{u_{1,t}}(x)\|^2\le L_0^2.
\end{equation}

In addition from~\cite[Lemma 2.1]{balasubramanian2019zeroth} it holds that,
\begin{equation}\label{eq:gnormsq}
\mathbb{E}[\|g_t-\nabla f_{u_{1,t}}(x)\|^2] \le \sigma_1 (n+5)+\sigma_2 u_{2,t}^2 (n+3)^3 = \sigma(u_{2,t},n) \le \bar{\sigma}
\end{equation}

Finally, we report another useful result.
\begin{lemma}\label{lem:almosthypo}
The following inequality holds: 
\[
\langle \nabla f_{u}(x)-\nabla f_{u}(y),x-y\rangle \ge -\rho\|x-y\|^2-4L_0 u \|x-y\|.
\]
\end{lemma}
\begin{proof} By using the definition of $f_{u_{1,t}}(x)$, we have
\[
\begin{array}{l}
\langle \nabla f_{u}(x)-\nabla f_{u}(y),x-y\rangle= \left\langle\nabla\left(\int_z \left(f(x+zu)-f(y+zu)\right)dz\right),x-y\right\rangle\\
\end{array}
\]
After a proper rewriting, we use \eqref{eq:hypocoer} to get a lower bound on the considered term:
 \[
\begin{array}{l}
\left\langle\left(\lim_{t\to 0}\frac{\int_z \left(f(x+zu+te_x)-f(x+zu)-f(y+zu+te_x)+f(y+zu)\right)}{t} dz\right)
,x-y\right\rangle
\\ \ge -\rho\|x-y\|^2+\\ 
+\left\langle\left(\lim_{t\to 0}\frac{\int_z \left(f(x+zu+te_x)-f(x+te_x)-f(x+zu)+f(x)-f(y+zu+te_x)+f(y+te_x)+f(y+zu)-f(y)\right)}{t} dz\right)
,x-y\right\rangle
\\ \qquad \ge -\rho\|x-y\|^2-4L_0u\|x-y\|,
\end{array}
\]
where the last inequality is obtained by considering the Lipschitz continuity of the function.
\end{proof}

The algorithmic scheme used in the paper is reported in Algorithm \ref{alg:PSDFA}, where with $\mbox{prox}_{\alpha r}$ we indicate the following
function:
\[
 \mbox{prox}_{\alpha r}(x)=\displaystyle\argmin_y \{r(y)+\frac{1}{2\alpha}\|y-x\|^2\}.
\]
At each iteration $t$ we simply build a two point estimate $g_t$ of the gradient related to the smoothed function and then apply 
a proximal map to the point  $x_t-\alpha_t g_t$, with $\alpha_t>0$ a suitably chosen stepsize. 

We let $\alpha_t$ be a diminishing step-size and set 
\begin{equation}\label{u12:vals}
u_{1,t}= \alpha_t^2\quad \mbox{and}\quad  u_{2,t} = \alpha_t^3.  
\end{equation}

\begin{algorithm}[h]
\caption{Proximal Stochastic Derivative Free Algorithm}
\label{alg:PSDFA}
  \begin{algorithmic}
  \par\vspace*{0.1cm}
  \item$\,\,\,$\hspace*{0.3truecm} \textbf{Input:} $x_0\in \R^n$, sequence $\{\alpha_t\}_{t\geq 0}$, and iteration count $T$.
     \par\vspace*{0.3cm}
 \item$\,\,\,$\hspace*{0.30cm} \textbf{For} $t = 0, 1, \dots, T$ 
    \par\vspace*{0.1cm}
 \item$\,\,\,$\hspace*{0.90cm} \textbf{Step 1)} Sample $\xi_t \sim P$, $Z_1\sim \mu_1$ and $Z_2\sim \mu_2$.
 \par\vspace*{0.1cm}
  \item$\,\,\,$\hspace*{0.90cm} \textbf{Step 2)} Set $u_{1,t}= \alpha_t^2$ and  $u_{2,t} = \alpha_t^3$.  
  \par\vspace*{0.1cm}
   \item$\,\,\,$\hspace*{0.90truecm} \textbf{Step 3)} Build the two point estimate $g_t=G(x_t,u_{1,t},u_{2,t},Z_{1,t},Z_{2,t},\xi_t)$.
   \par\vspace*{0.1cm}
   \item$\,\,\,$\hspace*{0.90truecm} \textbf{Step 4)} Set $x_{t+1}=\mbox{prox}_{\alpha_tr}(x_t-\alpha_t g_t)$.
  \par\vspace*{0.1cm}
 \item$\,\,\,$\hspace*{0.30cm} \textbf{End For}
  \par\vspace*{0.3cm}
 \item$\,\,\,$\hspace*{0.30cm} Sample $t^*\in\{0,\dots,T\}$ according to $\mathbb{P}(t^*=t)=\alpha_t/\sum_{i=0}^T \alpha_i$.
 \par\vspace*{0.1cm}
 \item$\,\,\,$\hspace*{0.30cm} \textbf{Return} $x^*_{t}$. 
  \end{algorithmic}
\end{algorithm}

We thus have in our scheme a derivative free version of  Algorithm 3.1 reported in \cite{davis2019stochastic}. 
\section{Convergence of the Derivative Free Algorithm}\label{s:convergence}
We now analyze the convergence properties of  Algorithm \ref{alg:PSDFA}. We follow~\cite[Section 3.2]{davis2019stochastic} 
in the proof of our results. We consider a value $\bar \rho> \rho$, and assume $\alpha_t<\min\left\{\frac{1}{\bar\rho},\frac{\bar\rho-\rho}{2}\right\}$ for all $t$.

We first define the function
\[
\phi^{u,t}(x) = f_{u_{1,t}}(x)+r(x),
\]
and introduce the Moreau envelope function
\[
\phi^{u,t}_{1/\lambda}(x) =\min_y \phi^{u,t}(y)+\frac{\lambda}{2}\|y-x\|^2\ ,
\]
with the proximal map 
\[
\text{prox}_{\phi^{u,t}/\lambda} (x)= \displaystyle\argmin_y \{\phi^{u,t}(y)+\frac{\lambda}{2}\|y-x\|^2\}. 
\]

We use the corresponding definition of $\phi_{1/\lambda}(x)$ as well in the convergence theory,
\[
\phi_{1/\lambda}(x) =\min_y \phi(y)+\frac{\lambda}{2}\|y-x\|^2 = \min_y f(y)+r(y)+\frac{\lambda}{2}\|y-x\|^2
\]
To begin with let 
\[
\hat x_t = \text{prox}_{\phi^{u,t}/\bar{\rho}} (x_t). 
\]
Some of the steps follow along the same lines given in~\cite[Lemma 3.5]{davis2019stochastic}, owing to the smoothness of $f_{u_{1,t}}(x)$.

We derive the following recursion lemma, which establishes an important descent property for the iterates.
\begin{lemma}\label{lem:recursion}
Let $\alpha_t$ satisfy,
\begin{equation}\label{alphacond}
\alpha_t \le \frac{\bar\rho-\rho}{(1+\bar\rho^2-2\bar\rho\rho+4\delta_0 L_0)}\ .
\end{equation}
where $\delta_0=1-\alpha_0 \bar\rho$.

Then it holds that there exists a $B$ independent of $t$ such that,
\[
\mathbb{E}_t\|x_{t+1}-\hat x_t\|^2 \le \|x_t-\hat x_t\|^2+\alpha^2_t B-\alpha_t(\bar\rho-\rho)\|x_t-\hat x_t\|^2.
\]
\end{lemma}
\begin{proof}
First we see that  $\hat x_t$ can be obtained as a proximal point of $r$:
\[
\begin{array}{l}
\alpha_t \bar{\rho} (x_t-\hat{x}_t)\in\alpha_t \partial r(\hat{x}_t)+\alpha_t \nabla f_{u_{1,t}}(\hat x_t)  
\Longleftrightarrow 
\\\\ \qquad 
\alpha_t \bar{\rho} x_t -\alpha_t \nabla f_{u_{1,t}}(\hat x_t)  +(1-\alpha_t \bar{\rho} )\hat{x}_t \in \hat{x}_t+\alpha_t \partial r(\hat{x}_t)
\\\\ \qquad
\Longleftrightarrow \hat x_t =\text{prox}_{\alpha_t r}\left(\alpha_t \bar{\rho} x_t-\alpha_t \nabla f_{u_{1,t}}(\hat x_t)+(1-\alpha_t \bar{\rho})\hat{x}_t\right).
\end{array}
\]
We notice that the last equivalence follows from the optimality conditions related to the proximal subproblem.
Letting $\delta_t=1-\alpha_t \bar\rho$, we get, 
\[
\begin{array}{l}
\mathbb{E}_t\|x_{t+1}-\hat x_t\|^2=\mathbb{E}_t\|\text{prox}_{\alpha_t r}(x_t-\alpha_t g_t)-\text{prox}_{\alpha_t r} (\alpha_t \bar\rho x_t
-\alpha_t \nabla f_{u_{1,t}}(x_t)+\delta_t \hat x_t)\|^2 \\
\quad \le \mathbb{E}_t\left\|x_t-\alpha_t g_t-(\alpha_t \bar\rho x_t-\alpha_t \nabla f_{u,1}(\hat x_t)+\delta_t \hat x_t)\right\|^2, 
\end{array}
\]
where the inequality is obtained by considering the non-expansiveness property of the proximal map $\text{prox}_{\alpha_t r}(x)$. We thus can write the following chain of equalities:
\[
\begin{array}{l}
\mathbb{E}_t\left\|x_t-\alpha_t g_t-(\alpha_t \bar\rho x_t-\alpha_t \nabla f_{u,1}(\hat x_t)+\delta_t \hat x_t)\right\|^2= \\
\quad = \mathbb{E}_t \left\|\delta_t(x_t-\hat x_t)-\alpha_t (g_t-\nabla f_{u_{1,t}}(\hat x_t) ) \right\|^2 =\\ 
\quad =\mathbb{E}_t \left\|\delta_t(x_t-\hat x_t)-\alpha_t (\nabla f_{u_{1,t}}(x_t)-\nabla f_{u_{1,t}}(\hat x_t) )-\alpha_t(g_t-\nabla f_{u_{1,t}}(x_t)) \right\|^2 = \\
\quad =\mathbb{E}_t \left\|\delta_t(x_t-\hat x_t)-\alpha_t (\nabla f_{u_{1,t}}(x_t)-\nabla f_{u_{1,t}}(\hat x_t) )\right\|^2
\\ \quad\quad-2\alpha_t\mathbb{E}_t
\left[\left\langle \delta_t(x_t-\hat x_t)-\alpha_t (\nabla f_{u_{1,t}}(x_t)-\nabla f_{u_{1,t}}(\hat x_t)) ,g_t-\nabla f_{u_{1,t}}(x_t)\right\rangle\right]
\\ \quad\quad+\alpha_t^2\mathbb{E}_t\left\|g_t-\nabla f_{u_{1,t}}(x_t)\right\|^2, \\
\end{array}
\]
with the first equality obtained by rearranging the terms inside the norm, the second one by simply adding and subtracting $\alpha_t \nabla f_{u_{1,t}}(x_t)$ to those terms,  and the third one
by taking into account the definition of euclidean norm and the basic properties of the expectation.  Now, we can get the following
\[
\begin{array}{l}
\mathbb{E}_t \left\|\delta_t(x_t-\hat x_t)-\alpha_t (\nabla f_{u_{1,t}}(x_t)-\nabla f_{u_{1,t}}(\hat x_t) )\right\|^2
\\ \quad\quad-2\alpha_t\mathbb{E}_t
\left[\left\langle \delta_t(x_t-\hat x_t)-\alpha_t (\nabla f_{u_{1,t}}(x_t)-\nabla f_{u_{1,t}}(\hat x_t)) ,g_t-\nabla f_{u_{1,t}}(x_t)\right\rangle\right]
\\ \quad\quad+\alpha_t^2\mathbb{E}_t\left\|g_t-\nabla f_{u_{1,t}}(x_t)\right\|^2\\
\quad = \left\|\delta(x_t-\hat x_t)-\alpha_t (\nabla f_{u_{1,t}}(x_t)-\nabla f_{u_{1,t}}(\hat x_t) )\right\|^2
\\ \quad\quad-2\alpha_t
\left[\left\langle \delta_t(x_t-\hat x_t)-\alpha_t (\nabla f_{u_{1,t}}(x_t)-\nabla f_{u_{1,t}}(\hat x_t)) ,\mathbb{E}[g_t]-\nabla f_{u_{1,t}}(x_t))\right\rangle\right]
\\ \quad\quad+\alpha_t^2\mathbb{E}_t\left\|g_t-\nabla f_{u_{1,t}}(x_t)\right\|^2 \\
\quad = \left\|\delta_t(x_t-\hat x_t)-\alpha_t (\nabla f_{u_{1,t}}(x_t)-\nabla f_{u_{1,t}}(\hat x_t) )\right\|^2
\\ \quad\quad-2\alpha_t
\left[\left\langle \delta_t(x_t-\hat x_t)-\alpha_t (\nabla f_{u_{1,t}}(x_t)-\nabla f_{u_{1,t}}(\hat x_t)) ,\frac{u_{2,t}}{u_{1,t}} e(x,u_{1,t},u_{2,t}))\right\rangle\right]
\\ \quad\quad+\alpha_t^2\mathbb{E}_t\left\|g_t-\nabla f_{u_{1,t}}(x_t)\right\|^2. 
\end{array}
\]
The first equality, in this case, was obtained by explicitly  taking expectation wrt to $\xi_t$, while we used equation \eqref{eq:ex1} to get the second one. 
We now try to upper bound the terms in the summation:
\[
\begin{array}{l}
\left\|\delta_t(x_t-\hat x_t)-\alpha_t (\nabla f_{u_{1,t}}(x_t)-\nabla f_{u_{1,t}}(\hat x_t) )\right\|^2
\\ \quad\quad-2\alpha_t
\left[\left\langle \delta_t(x_t-\hat x_t)-\alpha_t (\nabla f_{u_{1,t}}(x_t)-\nabla f_{u_{1,t}}(\hat x_t)) ,\frac{u_{2,t}}{u_{1,t}} e(x,u_{1,t},u_{2,t}))\right\rangle\right]
\\ \quad\quad+\alpha_t^2\mathbb{E}_t\left\|g_t-\nabla f_{u_{1,t}}(x_t)\right\|^2\\
\quad \leq \left\|\delta_t(x_t-\hat x_t)-\alpha_t (\nabla f_{u_{1,t}}(x_t)-\nabla f_{u_{1,t}}(\hat x_t) )\right\|^2
\\ \quad\quad+2\left(\alpha_t
\left\|\delta_t(x_t-\hat x_t)-\alpha_t (\nabla f_{u_{1,t}}(x_t)-\nabla f_{u_{1,t}}(\hat x_t))\right\|\right)
\left(\left|\frac{u_{2,t}}{u_{1,t}} \right|\left\|e(x,u_{1,t},u_{2,t}))\right\|\right)
\\ \quad\quad+\alpha_t^2\mathbb{E}_t\left\|g_t-\nabla f_{u_{1,t}}(x_t)\right\|^2  \\
\quad \le \left\|\delta_t(x_t-\hat x_t)-\alpha_t (\nabla f_{u_{1,t}}(x_t)-\nabla f_{u_{1,t}}(\hat x_t) )\right\|^2+\\
\quad\quad+\alpha^2_t\left\| \delta_t(x_t-\hat x_t)-\alpha_t (\nabla f_{u_{1,t}}(x_t)-\nabla f_{u_{1,t}}(\hat x_t)) \right\|^2
+\alpha_t^2 E^2
+\alpha_t^2\bar\sigma. \\
\end{array}
\]
The first inequality was obtained by using Cauchy-Schwarz.  
We then used the inequality $2a\cdot b\leq a^2+b^2$ combined with equation \eqref{u12:vals} and $\|e(x,u_{1,t},u_{2,t})\|\le E$  to upper bound the second term in the summation, and  
equation \eqref{eq:gnormsq} to upper bound the third term, thus getting the last inequality. Hence we can write
\[
\begin{array}{l}
\left\|\delta_t(x_t-\hat x_t)-\alpha_t (\nabla f_{u_{1,t}}(x_t)-\nabla f_{u_{1,t}}(\hat x_t) )\right\|^2+\\
\quad\quad +\alpha^2_t
\left\| \delta_t(x_t-\hat x_t)-\alpha_t (\nabla f_{u_{1,t}}(x_t)-\nabla f_{u_{1,t}}(\hat x_t)) \right\|^2
+\alpha_t^2 E^2
+\alpha_t^2\bar\sigma\\
\quad = (1+\alpha_t^2)\delta_t^2 \|x_t-\hat x_t\|^2-2(1+\alpha_t^2)\delta_t \alpha_t\langle x_t-\hat x_t,\nabla f_{u_{1,t}}(x_t)-\nabla f_{u_{1,t}} (\hat x_t)\rangle
\\\quad\quad+(1+\alpha_t^2)\alpha_t^2 \|\nabla f_{u_{1,t}}(x_t)-\nabla f_{u_{1,t}} (\hat x_t)\|^2+\alpha_t^2 E^2 
+\alpha_t^2\bar\sigma  \\
\quad \le (1+\alpha_t^2)\delta_t^2 \|x_t-\hat x_t\|^2+2(1+\alpha_t^2)\delta_t\alpha_t \rho\|x_t-\hat x_t\|^2+8(1+\alpha_t^2)\delta_t L_0 \alpha_t^3 \|x_t-\hat x_t\|\\\quad\quad+4(1+\alpha_t^2)\alpha_t^2 L_0^2
+\alpha_t^2 E^2
+\alpha_t^2\bar\sigma, \\
\end{array}
\]
where the equality is given by simply rearranging the terms in the summation and taking into account the definition of euclidean norm, and the inequality is 
obtained by upper bounding the scalar product by means of Lemma \ref{lem:almosthypo} and the third term in the summation by combining triangular inequality and equation~\eqref{eq:boundnabfu}.
Continuing:

\[
\begin{array}{l}
(1+\alpha_t^2)\delta_t^2 \|x_t-\hat x_t\|^2+2(1+\alpha_t^2)\delta_t\alpha_t \rho\|x_t-\hat x_t\|^2+8(1+\alpha_t^2)\delta_t L_0 \alpha_t^3 \|x_t-\hat x_t\|\\\quad\quad+4(1+\alpha_t^2)\alpha_t^2 L_0^2
+\alpha_t^2 E^2+\alpha_t^2\bar\sigma \\
\quad = (1+\alpha_t^2)\delta_t^2 \|x_t-\hat x_t\|^2+2(1+\alpha_t^2)\delta_t\alpha_t \rho\|x_t-\hat x_t\|^2+8(1+\alpha_t^2)\delta_t L_0 \left(\alpha_t^2\right)\left(\alpha_t \|x_t-\hat x_t\|\right)\\\quad\quad+4(1+\alpha_t^2)\alpha_t^2 L_0^2
+\alpha_t^2 E^2
+\alpha_t^2\bar\sigma \\
\quad \le (1+\alpha_t^2)\delta_t^2 \|x_t-\hat x_t\|^2+2(1+\alpha_t^2)\delta_t\alpha_t \rho\|x_t-\hat x_t\|^2+4(1+\alpha_t^2)\delta_t L_0 \alpha_t^4
\\\quad\quad+4(1+\alpha_t^2)\delta_t L_0\alpha^2_t \|x_t-\hat x_t\|^2+4(1+\alpha_t^2)\alpha_t^2 L_0^2
+\alpha_t^2 E^2
+\alpha_t^2\bar\sigma  \\
\quad = (1+\alpha_t^2)\delta_t^2 \|x_t-\hat x_t\|^2+2(1+\alpha_t^2)\delta_t\alpha_t \rho\|x_t-\hat x_t\|^2+4(1+\alpha_t^2)\delta_t L_0 \alpha_t^2 \|x_t-\hat x_t\|^2\\\quad\quad +4(1+\alpha_t^2)\delta_t L_0 \alpha_t^4+4(1+\alpha_t^2)\alpha_t^2 L_0^2
+\alpha_t^2 E^2
+\alpha_t^2\bar\sigma. \\
\end{array}
\]
The first and last equality are simply obtained by rearranging the terms in the summation. The inequality is obtained by upper
bounding the third term in the summation using inequality $2a\cdot b\leq a^2+b^2$. Finally, we have 
\[
\begin{array}{l}
(1+\alpha_t^2)\delta_t^2 \|x_t-\hat x_t\|^2+2(1+\alpha_t^2)\delta_t\alpha_t \rho\|x_t-\hat x_t\|^2+4(1+\alpha_t^2)\delta_t L_0 \alpha_t^2 \|x_t-\hat x_t\|^2\\\quad\quad +4(1+\alpha_t^2)\delta_t L_0 \alpha_t^4+4(1+\alpha_t^2)\alpha_t^2 L_0^2
+\alpha_t^2 E^2
+\alpha_t^2\bar\sigma \\
\quad = \left[1-2\alpha_t\bar\rho+\alpha^2_t\bar\rho^2+\alpha_t^2-2\alpha^3_t\bar\rho+\alpha^4_t\bar\rho^2
+2\alpha_t \rho-2\alpha^2_t\bar\rho\rho+2\alpha^3_t \rho-2\alpha^4_t\bar\rho\rho\right.\\\quad\quad\left.+4\delta_t L_0 \alpha_t^2+4\delta_t L_0 \alpha_t^4\right] \|x_t-\hat x_t\|^2\\\quad\quad +\left[4(1+\alpha_t^2)\delta_t L_0 \alpha_t^2+4(1+\alpha_t^2)L_0^2
+E^2+\bar\sigma\right]\alpha_t^2 
\\
\quad = \left[1-2\alpha_t(\bar\rho-\rho)+\alpha^2_t(1+\bar\rho^2-2\bar\rho\rho+4\delta_t L_0)-2\alpha^3_t(\bar\rho-\rho)+\alpha^4_t(\bar\rho^2
-2\bar\rho\rho+4\delta_t L_0 )\right] \|x_t-\hat x_t\|^2\\\quad\quad +\left[4(1+\alpha_t^2)\delta_t L_0 \alpha_t^2+4(1+\alpha_t^2)L_0^2
+E^2+\bar\sigma\right]\alpha_t^2 
\\ \quad \le \left[1-\alpha_t(\bar\rho-\rho)\right]\|x_t-\hat x_t\|^2+B\alpha_t^2,
\end{array}
\]
where the last inequality is obtained by simply considering the expression of $\alpha_t$ in equation \eqref{alphacond}.
\end{proof}
After proving Lemma \ref{lem:recursion}, we can now state the main convergence result for Algorithm~\ref{alg:PSDFA}.
\begin{theorem}\label{th:conv}
The sequence generated by Algorithm~\ref{alg:PSDFA} satisfies,
\[
\begin{array}{l}
\mathbb{E}[\phi_{1/\bar{\rho}}(x_{t+1})] \le \mathbb{E}[\phi_{1/{\bar\rho}}(x_t)]+\alpha_t^2\bar B\sqrt{n+2}+\frac{B\bar\rho}{2}\alpha_t^2-\frac{\alpha_t(\bar\rho-\rho)}{2\bar\rho}\mathbb{E}[\|\nabla \phi^{u,t}_{1/{\bar{\rho}}}(x_t)\|^2]\\
\end{array}
\]
and thus,
\[
\begin{array}{l}
\mathbb{E}[\|\nabla \phi^{u,t^*}_{1/{\bar{\rho}}}(x_{t^*})\|^2] = \frac{1}{\sum_{t=0}^T \alpha_t} \sum\limits_{t=0}^T \alpha_t \mathbb{E}[\|\nabla \phi^{u,t}_{1/{\bar{\rho}}}(x_t)\|^2]
\le\\ 
\qquad\qquad\qquad\qquad\le\frac{2\bar\rho}{\bar\rho-\rho} \frac{\phi^{u,0}(x_0)-\min \phi +(\bar B\sqrt{n+2}+\frac{B\bar\rho}{2})\sum\limits_{t=0}^T\alpha_t^2}{\sum\limits_{t=0}^T \alpha_t}.
\end{array}
\]


\end{theorem}
\begin{proof}

We have,
\[
\begin{array}{l}
\mathbb{E}_t[\phi_{1/\bar{\rho}}(x_{t+1})] \le \mathbb{E}_t\left[\phi(\hat x_t)+\frac{\bar\rho}{2}\|\hat x_t-x_{t+1}\|^2\right] \\
\qquad \le \phi(\hat x_t)+\frac{\bar\rho}{2}\left(\|x_t-\hat x_t\|^2+B\alpha_t^2-\alpha_t(\bar\rho-\rho)\|x_t-\hat x_t\|^2\right) \\
\qquad \le \phi^{u,t}(\hat x_t) +\frac{\bar\rho}{2}\left(\|x_t-\hat x_t\|^2+B\alpha_t^2-\alpha_t(\bar\rho-\rho)\|x_t-\hat x_t\|^2\right), \\
\end{array}
\]
where the first inequality comes from the definition of the proximal map, the second by considering the result proved in Lemma \ref{lem:recursion},
and the third by taking into account the first inequality in equation \eqref{eq:smoothacc}. 
\[
\begin{array}{l}
 \phi^{u,t}(\hat x_t) +\frac{\bar\rho}{2}\left(\|x_t-\hat x_t\|^2+B\alpha_t^2-\alpha_t(\bar\rho-\rho)\|x_t-\hat x_t\|^2\right)=\\
\qquad = \phi^{u,t}_{1/{\bar\rho}}(x_t)+\frac{B\bar\rho}{2}\alpha_t^2-\frac{\bar\rho\alpha_t}{2}(\bar\rho-\rho)\|x_t-\hat x_t\|^2 \le  \\
\qquad \le \phi_{1/{\bar\rho}}(x_t)+u_{1,t}\bar B\sqrt{n+2}+\frac{B\bar\rho}{2}\alpha_t^2-\frac{\bar\rho\alpha_t}{2}(\bar\rho-\rho)\|x_t-\hat x_t\|^2 \le \\
\qquad \le \phi_{1/{\bar\rho}}(x_t)+\alpha_t^2\bar B\sqrt{n+2}+\frac{B\bar\rho}{2}\alpha_t^2-\frac{\bar\rho\alpha_t}{2}(\bar\rho-\rho)\|x_t-\hat x_t\|^2= \\
\qquad = \phi_{1/{\bar\rho}}(x_t)+\alpha_t^2\bar B\sqrt{n+2}+\frac{B\bar\rho}{2}\alpha_t^2-\frac{\alpha_t(\bar\rho-\rho)}{2\bar\rho}\|\nabla \phi^{u,t}_{1/{\bar{\rho}}}(x_t)\|^2,
\end{array}
\]
with the first inequality obtained by taking into account the second inequality in equation \eqref{eq:smoothacc}, and the second by using 
definition of $u_{1,t}$ in \eqref{u12:vals}. Now, we take full expectations and obtain:



\[
\begin{array}{l}
\mathbb{E}[\phi_{1/\bar{\rho}}(x_{t+1})] \le \mathbb{E}[\phi_{1/{\bar\rho}}(x_t)]+\alpha_t^2\bar B\sqrt{n+2}\\
\qquad\qquad\qquad\qquad+\frac{B\bar\rho}{2}\alpha_t^2-\frac{\alpha_t(\bar\rho-\rho)}{2\bar\rho}\mathbb{E}[\|\nabla \phi^{u,t}_{1/{\bar{\rho}}}(x_t)\|^2].
\end{array}
\]

The rest of the proof is as in~\cite[Theorem 3.4]{davis2019stochastic}. In particular, summing the recursion, we get,
\[
\begin{array}{l}
\mathbb{E}[\phi_{1/\bar{\rho}}(x_{T+1})] \le \mathbb{E}[\phi^{u,0}_{1/{\bar\rho}}(x_0)]+(\bar B\sqrt{n+2}+\frac{B\bar\rho}{2})\sum\limits_{t=0}^T\alpha_t^2\\
\qquad\qquad\qquad\qquad-\frac{(\bar\rho-\rho)}{2\bar\rho}\sum\limits_{t=0}^T \alpha_t\mathbb{E}[\|\nabla \phi^{u,t}_{1/{\bar{\rho}}}(x_t)\|^2].
\end{array}
\]
Now, noting that 
\[
\phi^{u,t}(x)=\min_y f^{u,t}(y)+r(y)+\frac{\lambda}{2}\|y-x\|^2\ge\min_y \min f+r(y)+\frac{\lambda}{2}\|y-x\|^2 \ge \min \phi,
\]
we can finally state that,
\[
\begin{array}{l}
\frac{1}{\sum_{t=0}^T \alpha_t} \sum\limits_{t=0}^T \alpha_t \mathbb{E}[\|\nabla \phi^{u,t}_{1/{\bar{\rho}}}(x_t)\|^2]
\le\\ 
\qquad\qquad\qquad\qquad\le\frac{2\bar\rho}{\bar\rho-\rho} \frac{\phi^{u,0}(x_0)-\min \phi +(\bar B\sqrt{n+2}+\frac{B\bar\rho}{2})\sum\limits_{t=0}^T\alpha_t^2}{\sum\limits_{t=0}^T \alpha_t}.
\end{array}
\]

Since the left-hand side is by definition $\mathbb{E}[\|\nabla \phi^{u,t^*}_{1/{\bar{\rho}}}(x_{t^*})\|^2]$, we get the final result.

\end{proof}

\section{Numerical Results}\label{s:numerics}
In this section, we investigate the numerical performance of Algorithm~\ref{alg:PSDFA} on a set of standard weakly convex optimization
problems defined in~\cite{davis2019stochastic}. In particular we consider phase retrieval, which seeks to minimize the function,
\begin{equation}\label{eq:phase}
\min_{x\in\mathbb{R}^d} \frac{1}{m} \sum\limits_{i=1}^m \left|\langle a_i,x\rangle^2-b_i\right|
\end{equation}
and blind deconvolution, which seeks to minimize 
\begin{equation}\label{eq:blind}
\min_{(x,y)\in\mathbb{R}^d} \frac{1}{m} \sum\limits_{i=1}^m \left|\langle u_i,x\rangle\langle v_i,y\rangle-b_i\right|
\end{equation}
We generate random Gaussian measurements in $N(0,I_{d\times d}$) and 
a target signal $\bar{x}$ uniformly on the random sphere to compute $b_i$ with dimensions $(d,m)=(10,30),(20,60),(40,120)$.
We compare Algorithm~\ref{alg:PSDFA} to the stochastic subgradient method and the stochastic proximal
method in~\cite{davis2019stochastic}. We generate ten runs of each for every dimension and pick the best one according to the
final objective value. The total number of iterations used in all cases is $1000m$.

For phase retrieval we generate $\alpha_0$ uniformly from \texttt{[1e-5,1e-4]}. We show the path of the objective values
in Figures~\ref{fig:phase1},~\ref{fig:phase2} and~\ref{fig:phase3}.

For blind deconvolution we generate $\alpha_0$ uniformly from \texttt{[1e-6,1e-3]}. We show the path of the objective values
in Figures~\ref{fig:blind1},~\ref{fig:blind2} and~\ref{fig:blind3}.

It is interesting that for the smaller dimensional problems, the zero order algoritm performs on par with the ones that use
the stochastic subgradient oracle, sometimes even outperforming them (probably owing to the additional noise, thus one out of
the ten runs can become a favorable outlier). For the largest dimension it is more evident that the performance of the
derivative-free algorithm is slower, albeit still convergent. 

\begin{center}
\begin{figure}
\includegraphics[scale=0.7,trim=3cm 9cm 3cm 9.5cm,clip,width=\textwidth]{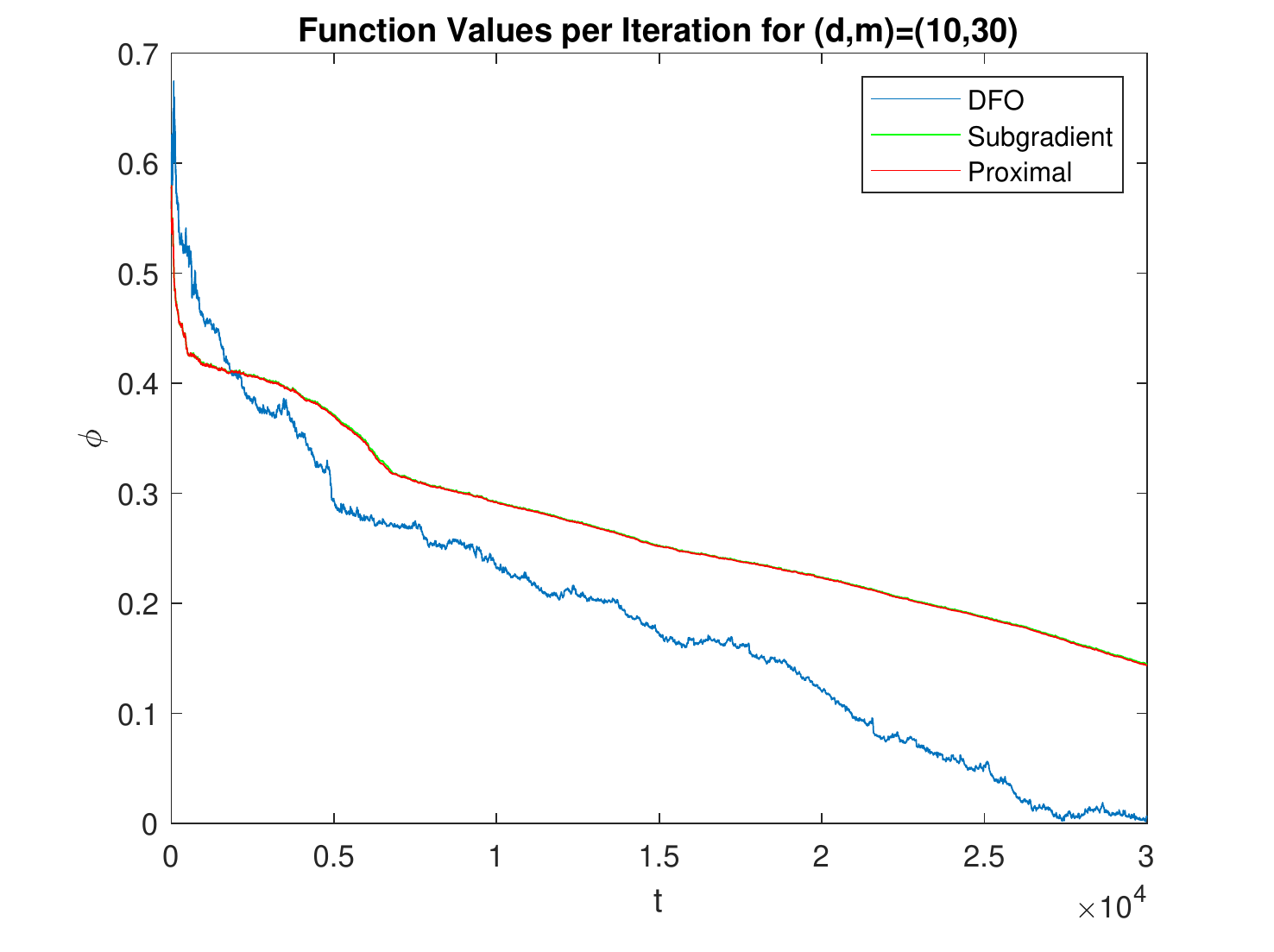}
\caption{Convergence of the function values for phase retrieval,~\eqref{eq:phase}. Interestingly, in this case, DFO is able to outperform the (sub)derivative based methods, in particular exhibiting an accelerate convergence rate in the early iterations~\label{fig:phase1}}
\end{figure}
\end{center} 

\begin{center}
\begin{figure}
\includegraphics[scale=0.7,trim=3cm 9cm 3cm 9.5cm,clip,width=\textwidth]{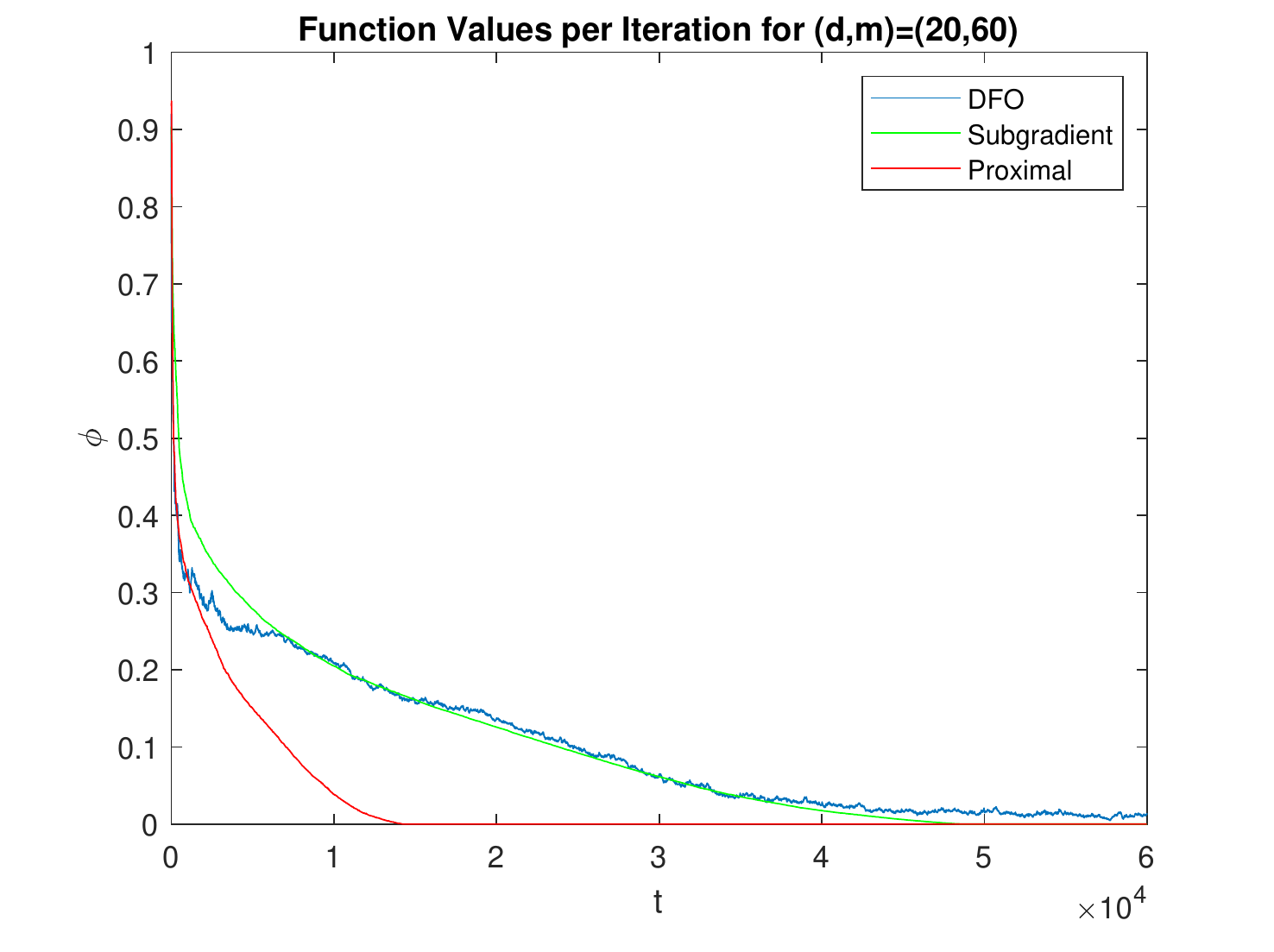}
\caption{Convergence of the function values for phase retrieval,~\eqref{eq:phase}.
In this case, DFO is slower than the Proximal method, but still competitive with the Subgradient method.~\label{fig:phase2}}
\end{figure}
\end{center} 

\begin{center}
\begin{figure}
\includegraphics[scale=0.7,trim=3cm 9cm 3cm 9.5cm,clip,width=\textwidth]{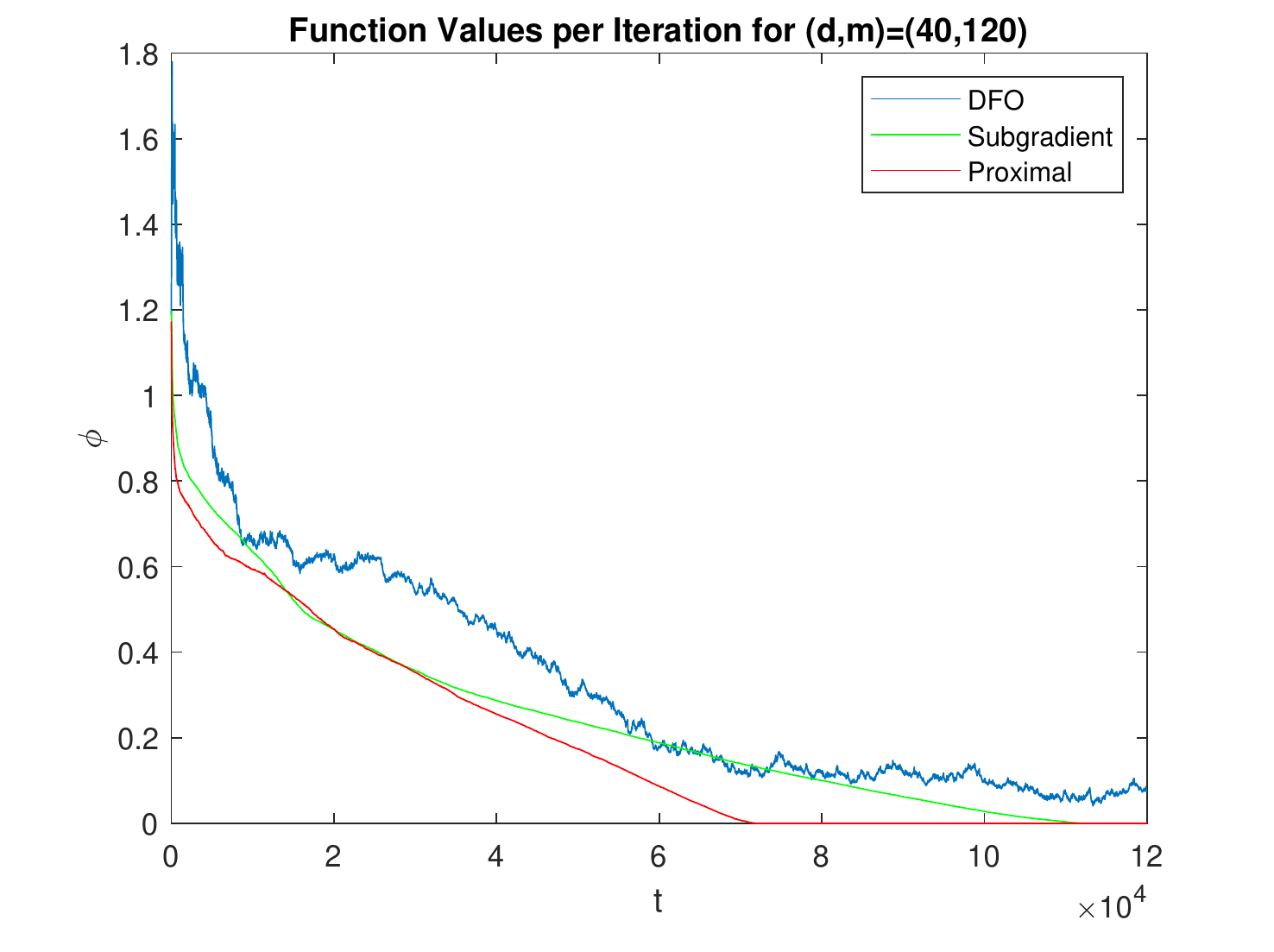}
\caption{Convergence of the function values for phase retrieval,~\eqref{eq:phase}.
For the larger scale case, although the DFO algorithm is still asymptotically convergent, it is seems to be slightly slower than the (sub)derivative based methods.~\label{fig:phase3}}
\end{figure}
\end{center} 

\begin{center}
\begin{figure}
\includegraphics[scale=0.7,trim=3cm 9cm 3cm 9.5cm,clip,width=\textwidth]{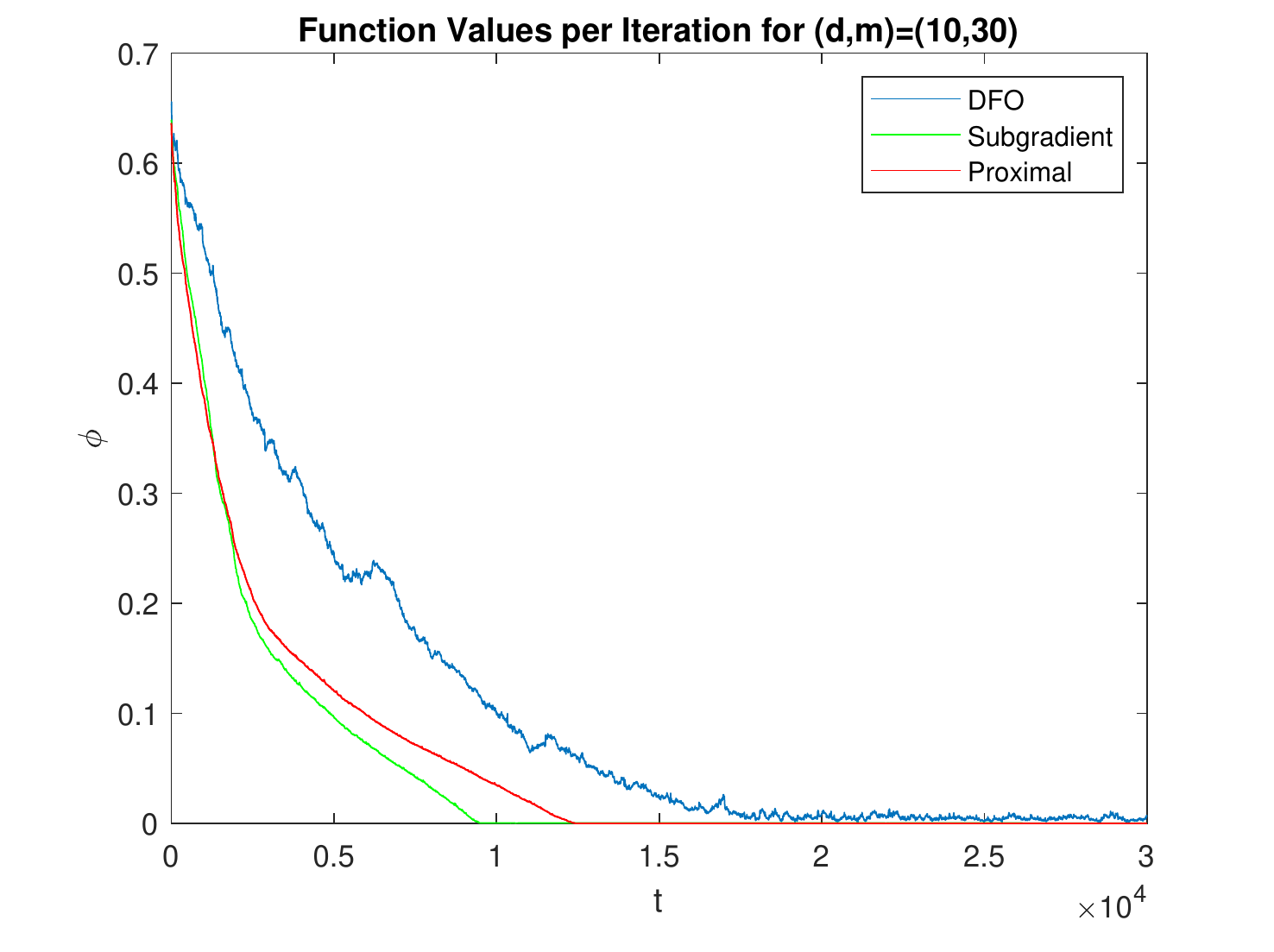}
\caption{Convergence of the function values for blind deconvolution,~\eqref{eq:blind}. In this case, DFO is competitive with the (sub)derivative based methods.~\label{fig:blind1}}
\end{figure}
\end{center} 

\begin{center}
\begin{figure}
\includegraphics[scale=0.7,trim=3cm 9cm 3cm 9.5cm,clip,width=\textwidth]{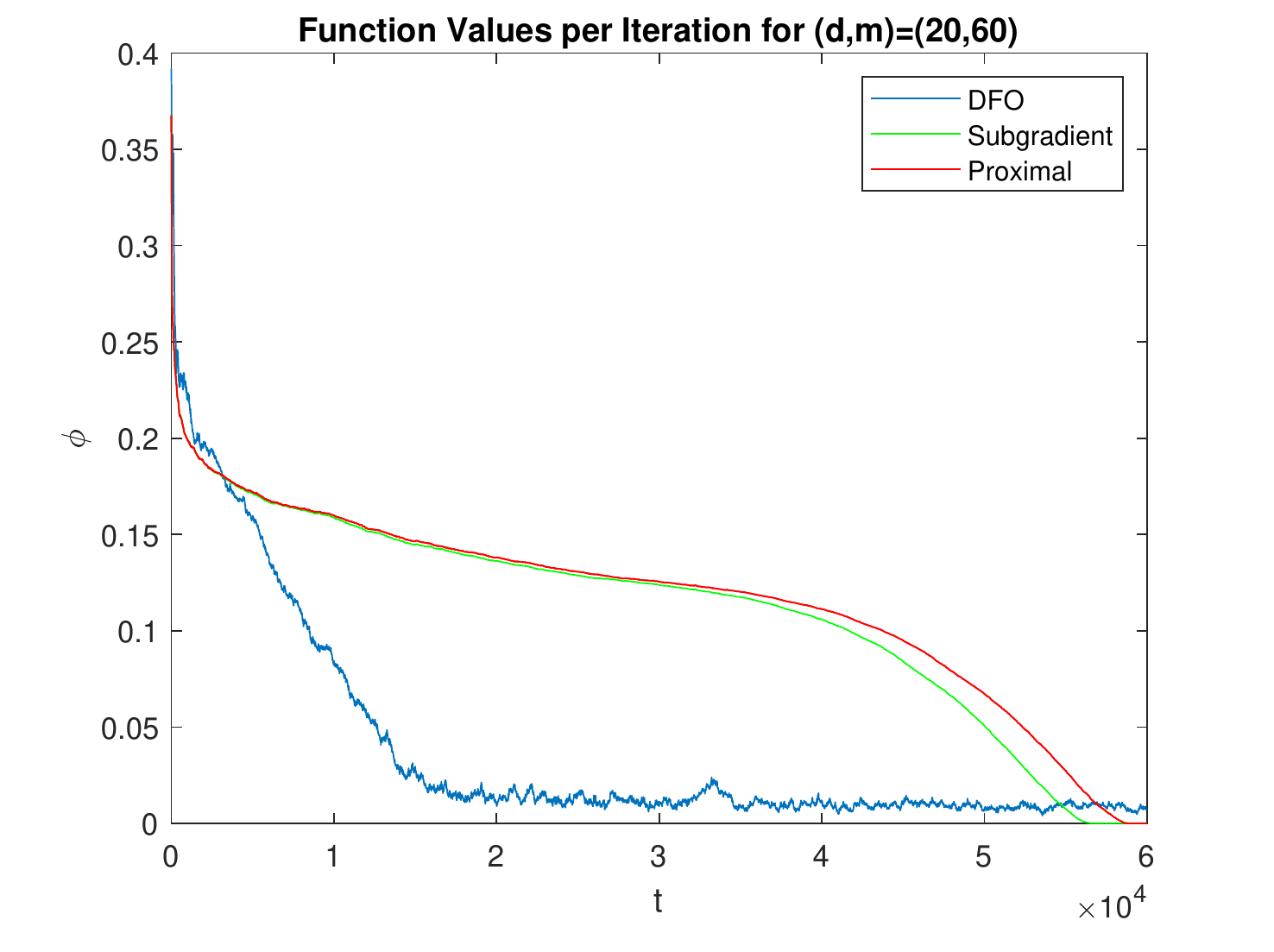}
\caption{Convergence of the function values for blind deconvolution,~\eqref{eq:blind}.
Interestingly, in this case, DFO, albeit being a bit noisy, decreases faster than the (sub)derivative based methods.~\label{fig:blind2}}
\end{figure}
\end{center} 

\begin{center}
\begin{figure}
\includegraphics[scale=0.7,trim=3cm 9cm 3cm 9.5cm,clip,width=\textwidth]{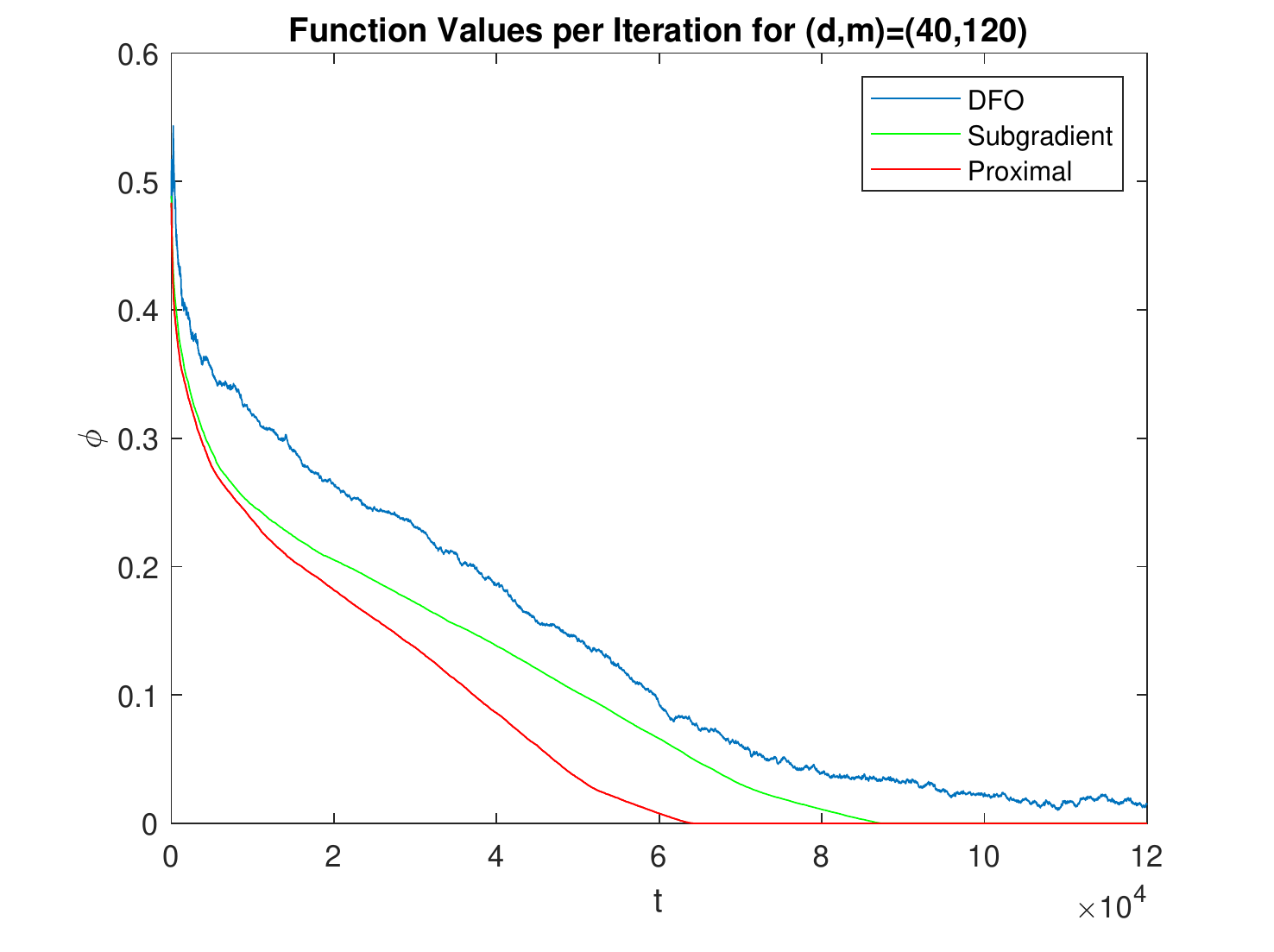}
\caption{Convergence of the function values for blind deconvolution,~\eqref{eq:blind}.
For the larger scale case, although the DFO algorithm is still asymptotically convergent, it is seems to be again slightly slower than the (sub)derivative based methods.~\label{fig:blind3}}
\end{figure}
\end{center}

\section{Conclusion}\label{s:conclusion}

In this paper we studied, for the first time, minimization of a stochastic weakly convex function without
the presence of an oracle of a noisy estimate of the subgradient of the function, i.e., in the context
of derivative-free or zero order optimization. We were able to derive theoretical convergence rate results
on par with the standard methods for stochastic weakly convex optimization, and demonstrated the algorithm's
efficacy on a couple of standard test cases. In expanding the scope of zero order optimization, we hope that this
work highlights the potential of derivative free methods in general, and the two point smoothed function
approximation technique in particular, to an increasingly wider class of problems. 

\bibliographystyle{plain}
\bibliography{refs}

\end{document}